\documentclass[reqno]{amsart}
\usepackage{amsbsy}
\usepackage{amsfonts}
\usepackage{amsmath}
\usepackage{amsthm}
\usepackage{amssymb}
\usepackage{enumitem}
\usepackage{ytableau}

\usepackage{hyperref}
\usepackage{color}
\usepackage{etoolbox}

\theoremstyle{plain}

\newtheorem{thm}{Theorem}[section]
\newtheorem{lem}[thm]{Lemma}

\newtheorem{conj}[thm]{Conjecture}

\newcommand\MOD{\textrm{ (mod }}

\newcommand\Tr{\mathrm{Tr}}

\newcommand\bb{\mathbb}

\newcommand\primesum{\sideset{}{'}\sum}
\newcommand\primeprod{\sideset{}{'}\prod}

\begin{document}

\title{The covariance of almost-primes in $\mathbb{F}_q[T]$}
\author{Brad Rodgers}
\address{Institut f\"{u}r Mathematik, Universit\"{a}t Z\"{u}rich, Winterthurerstr. 190, CH-8057 Z\"{u}rich}
\email{brad.rodgers@math.uzh.ch}
\date{}

\begin{abstract}
We estimate the covariance in counts of almost-primes in $\bb F_q[T]$, weighted by higher-order von Mangoldt functions. The answer takes a pleasant algebraic form. This generalizes recent work of Keating and Rudnick that estimates the variance of primes, and makes use, as theirs, of a recent equidistribution result of Katz. 

In an appendix we prove some related identities for random matrix statistics, which allows us to give a quick proof of a $2\times2$ ratio theorem for the characteristic polynomial of the unitary group. We additionally identify arithmetic functions whose statistics mimic those of hook Schur functions.
\end{abstract}


\maketitle

\section{Introduction}

\subsection{} 
The purpose of this note is to generalize a recent estimate of Keating and Rudnick \cite{KeRu} for the variance in counts of primes in $\bb F_q[T]$. Here $\bb F_q[T]$ is the ring of polynomials with coefficients in the finite field of $q$ elements. We generalize their result to an estimate for the covariance in counts of almost-primes. 

An `almost-prime' designates for our purposes an element whose factorization contains no more than $j$ distinct prime (that is, irreducible) factors, for a fixed number $j$, referred to as the order. Such almost-primes may be weighted by a function field analogue of higher-order von Mangoldt functions (famously used by Selberg and Erd\H{o}s in an elementary proof of the prime number theorem). What is especially curious is that weighted in this manner, the covariance of almost-primes of different orders takes on a simple and pleasant algebraic form. 

We use the estimate over $\bb F_q[T]$ to provide evidence for a conjecture made in the author's thesis for the covariance of counts of almost-primes among the integers.

\subsection{}
We begin with a brief review of the situation in the integers. Here the prime number theorem tells us that the count of prime numbers between $1$ and $x$ is asymptotic to $x/\log x$, or equivalently that $\psi(x) \sim x$, where as usual
$$
\psi(x):= \sum_{n\leq x} \Lambda(n),
$$
and $\Lambda(n)$ is defined to be $\log p$ if $n=p^k$ a prime power, and $0$ otherwise.

It follows that if we define a count of primes around $x$ in a short interval of size $H$ by
$$
\psi(x;H):=\sum_{x< n \leq x+H} \Lambda(n),
$$
then $\psi(x;H)$ will take roughly the value $H$, for $x$ taken on average. 

Building on work of Gallagher and Mueller \cite{GaMu}, the variance of these counts was studied by Goldston and Montgomery \cite{GoMo}, who show that the following conjecture is equivalent to a strong form of Montgomery's pair correlation conjecture for the zeros of the zeta function.

\begin{conj}[Goldston-Montgomery]
\label{Prime_var}
Define
\begin{equation}
\label{prime_oscil}
\widetilde{\psi}(x;H) = \psi(x;H)-H.
\end{equation}
For any fixed $\epsilon > 0$, 
\begin{equation}
\label{prime_var}
\frac{1}{X}\int_0^X \widetilde{\psi}(x;H)^2\,dx \sim H (\log X - \log H),
\end{equation}
uniformly for $H\in [1,X^{1-\epsilon}]$.
\end{conj}

\subsection{}
A function field analogue of this conjecture was recently proved by Keating and Rudnick, with an equidistribution result of Katz \cite{Ka} playing the crucial role that Montgomery's pair correlation conjecture plays in Conjecture \ref{Prime_var}.

Over $\bb F_q[T]$,  the von Mangoldt function is defined for a polynomial $f$ by
$$
\Lambda(f):= \begin{cases}
    \deg P, & \text{if $f=cP^k$ for $c\in \bb F_q$ and $P$ an irreducible polynomial}.\\
    0, & \text{otherwise}.
  \end{cases}
$$
For notational reasons, throughout this paper we let $\mathcal{M}_n\subset \mathbb{F}_q[T]$ denote the collection of monic polynomials of degree $n$. The prime number theorem in $\bb F_q[T]$ is that
\begin{equation}
\label{PNT_fin}
\sum_{f\in\mathcal{M}_n}\Lambda(f) = q^n.
\end{equation}
The reader should check that
$$
\sum_{f\in\mathcal{M}_n} 1 = q^n.
$$
This information is summarized by the zeta function\footnote{The function 
$$
\zeta_{\bb F_q[T]}(s) := \mathcal{Z}(q^{-s}) = \primesum \frac{1}{|f|^s}
$$
may make more clear the analogy to the classical Riemann zeta function.} over $\bb F_q[T]$, defined for $|u|<1/q$ by
\begin{align}
\label{zeta_def}
\notag \mathcal{Z}(u) &:=\primeprod_{P} \frac{1}{1-u^{\deg (P)}} \\
&= \primesum_{f} u^{\deg(f)} \\
\notag &= \frac{1}{1-qu},
\end{align}
where above and in what follows a primed sum or product indicates an additional restriction to monic polynomials; thus the Euler product above is over all monic irreducible polynomials, and likewise the primed sum is over all monic polynomials.

In setting up a $\bb Z$-to-$\bb F_q[T]$ dictionary, the degree of a polynomial is akin to the logarithm of an integer, while $|f|:=q^{\deg f}$ may be thought of as corresponding to the absolute-value of an integer. By analogy, a short-interval containing a polynomial $f$ of degree $n$ is defined to be the collection of polynomials
$$
I(f;h) := \{g: \deg(f-g) \leq h\} = \{g: |f-g| \leq q^h\},
$$
where $h < n$. (Note that for $f$ monic, the short interval $I(f;h)$ defined above will consist exclusively of monic polynomials.)

We define, for a polynomial $f$, the count of primes in a short interval around $f$,
$$
\Psi(f;h):= \sum_{g\in I(f;h)} \Lambda(g).
$$
and also define the count minus its average value,
$$
\widetilde{\Psi}(f;h):= \sum_{g\in I(f;h)} (\Lambda(g)-1).
$$
(Note that $\sum_{g\in I(f;h)} 1 = q^{h+1}.$)

With this set up, we have

\begin{thm}[Keating-Rudnick]
\label{Prime_var_fin}
For fixed $0 \leq h \leq n-4$,
$$
\lim_{q\rightarrow\infty} \frac{1}{q^{h+1}} \bigg(\frac{1}{q^n}\sum_{f\in \mathcal{M}_n} \widetilde{\Psi}(f;h)^2\bigg) = n-h-2.
$$
\end{thm}

As the authors note, this may be compared to Conjecture \ref{Prime_var} with the dictionary\footnote{That the variance in Theorem \ref{Prime_var_fin} is $n-h-2$ rather than $n-h-1$ may seem surprising at first, but one must keep in mind that the expression  \eqref{prime_oscil} in Conjecture \ref{Prime_var} is only an asymptotic formula, and when $n-h$ is large, $n-h-1\sim n-h-2$. Indeed, there is a refinement of Conjecture \ref{Prime_var} due to Montgomery and Soundararajan \cite{MoSo2} in which the left hand side of \eqref{prime_oscil} is conjectured to be
$$
H(\log X - \log H - (\gamma+\log 2\pi) + o(1)),
$$
in closer agreement with Theorem \ref{Prime_var_fin}.}
\begin{equation}
\label{dictionary}
X \leftrightarrow q^n, \quad H \leftrightarrow q^{h+1}, \quad  \log X \leftrightarrow n, \quad \log H \leftrightarrow h+1.
\end{equation}

\subsection{}
We generalize the variance estimate in Theorem \ref{Prime_var_fin} to an estimate for the covariance of counts of almost-primes. We count almost-primes of order $j$ by means of higher-order von Mangoldt functions, defined inductively in the following way:
\begin{equation}
\label{highvm_1}
\Lambda_1(f):= \Lambda(f)
\end{equation}
and for $j\geq 2$,
\begin{equation}
\label{highvm_2}
\Lambda_j(f):= \primesum_{d|f}\Lambda_{j-1}(d)\Lambda(f/d) + \Lambda_{j-1}(f)\deg(f),
\end{equation}
where the sum is over only monic polynomials $d$ dividing $f$. It is evident from this definition that $\Lambda_j$ is supported on polynomials that have no more than $j$ distinct irreducible factors. It is useful to also define $\Lambda_0(f)$ to be $1$ if $f\in \bb F_q$ and $0$ otherwise. Note that for a unit $c\in \bb F_q^\times$, we have $\Lambda_j(cf) = \Lambda_j(f)$.

$\Lambda_j$ can also be written more succinctly
\begin{equation}
\label{Lambda_j}
\Lambda_j(f) = \primesum_{d|f} \mu(d) \deg^j(f/d)
\end{equation}
where $\mu$ is the M\"{o}bius function, $\mu(d):=(-1)^k$ if $d = cP_1\cdots P_k$, for distinct irreducible polynomials $P_i\in \bb F_q[T]$ and $c\in \bb F_q$; and $\mu(d)=0$ otherwise.

$\Lambda_j$ is not a multiplicative function, but from \eqref{Lambda_j}, one may prove that for $f$ and $g$ coprime polynomials,
\begin{equation}
\label{Lambda_fac}
\Lambda_j(fg) = \sum_{\ell=0}^j {j \choose \ell} \Lambda_{\ell}(f) \Lambda_{j-\ell}(g).
\end{equation}

Clearly $\Lambda_j(f)$ is never negative. From \eqref{Lambda_j}, we see that $\primesum_{d|f} \Lambda_j(d) = \deg^j(f)$, and so we have also the upper bound,
\begin{equation}
\label{vm_bound}
\Lambda_j(f) \leq \deg^j(f).
\end{equation}

By using the definition \eqref{highvm_1} and \eqref{highvm_2} of $\Lambda_j$ and inducting on the relation \eqref{PNT_fin}, one may verify that
\begin{equation}
\label{vm_average}
\sum_{f\in \mathcal{M}_n} \Lambda_j(f) = q^n(n^j-(n-1)^j),
\end{equation}
so that
$$
\widetilde{\Lambda}_j(f):= \Lambda_j(f) - (n^j-(n-1)^j),\quad \textrm{for $n := \deg(f)$},
$$
is on average $0$. We define\footnote{Note that here, our notational convention for $\widetilde{\Psi}$ differs from that of \cite{KeRu}.}
$$
\widetilde{\Psi}_j(f;h):= \sum_{g\in I(f;h)} \widetilde{\Lambda}_j(g),
$$
which is a count in a short interval around $f$ of almost-primes minus the average value of the count.

Our main result is as follows:
\begin{thm}
\label{Covar_fin}
For $j,k \geq 1$,
\begin{equation}
\label{Covar_1_fin}
\lim_{q\rightarrow\infty} \frac{1}{q^n}\sum_{g\in \mathcal{M}_n} \Lambda_j(f)\Lambda_k(f) = \sum_{d=1}^n (d^j-(d-1)^j)(d^k-(d-1)^k),
\end{equation}
\begin{equation}
\label{Covar_2_fin}
\lim_{q\rightarrow\infty}\frac{1}{q^n} \sum_{f\in\mathcal{M}_n} \widetilde{\Lambda}_j(f)\widetilde{\Lambda}_k(f) = \sum_{d=1}^{n-1} (d^j-(d-1)^j)(d^k-(d-1)^k),
\end{equation}
and for $0\leq h\leq n-4$,
\begin{equation}
\label{Covar_3_fin}
\lim_{q\rightarrow\infty}\frac{1}{q^{h+1}}\bigg(\frac{1}{q^n} \sum_{f\in\mathcal{M}_n} \widetilde{\Psi}_j(f;h)\widetilde{\Psi}_k(f;h)\bigg) = \sum_{d=1}^{n-h-2} (d^j-(d-1)^j)(d^k-(d-1)^k).
\end{equation}
\end{thm}

Here \eqref{Covar_1_fin} and \eqref{Covar_2_fin} are elementary, and we give a short and self-contained proof of them below.\footnote{We note in passing that it is very natural to make the formal identification $\widetilde{\Psi}(f;-1) := \widetilde{\Lambda}_j$(f) in Theorem \ref{Covar_fin}.} \eqref{Covar_3_fin} on the other hand lies deeper, and depends on the result of Katz discussed below.

\subsection{}

Theorem \ref{Covar_fin} provides support for a generalization, made in the authors thesis \cite{Ro2, Ro1}, of Conjecture \ref{Prime_var}. We recall the conjecture here. Define $\Lambda_j$ over the integers in the same way as above, so that
$$
\Lambda_1(n):=\Lambda(n),
$$
$$
\Lambda_j(n) := \sum_{d\delta = n} \Lambda_{j-1}(d)\Lambda_1(\delta) + \log(n) \Lambda_{j-1}(n).
$$
Or alternatively,
$$
\Lambda_j(n) = \sum_{d\delta = n} \mu(d) \log^j \delta.
$$
As in the function field case, $\Lambda_j$ is supported on integers with no more than $j$ distinct prime divisors. It is well known (see \cite[pp. 299]{Iv} for instance) that
$$
\psi_j(x):= \sum_{n \leq x} \Lambda_j(n) = x P_{j-1}(\log x) + o(x)
$$
where $P_{j-1}$, defined by this relation, is a $j-1$ degree polynomial with 
$$
P_{j-1}(\log x) = j(\log x)^{j-1}+ o(\log^{j-1} x).
$$
We define
$$
\widetilde{\psi}_j(x):= \psi_j(x) - x P_{j-1}(\log x),
$$
and
$$
\widetilde{\psi}_j(x;H) := \widetilde{\psi}_j(x+H)-\widetilde{\psi}_j(x).
$$
Note that $\widetilde{\psi}_1(x;H) = \widetilde{\psi}(x;H)$. In general $\widetilde{\psi}_j(x;H)$ is a count of almost-primes in the interval $(x,x+H]$ minus a regular approximation to this count. By analogy with Theorem \ref{Covar_fin}, using the dictionary \eqref{dictionary} and the fact that
$$
\sum_{d=1}^{n-h-2}(d^j-(d-1)^j)(d^k-(d-1)^k) = \frac{jk}{j+k-1} (n-h)^{j+k-1} + O((n-h)^{j+k-2}),
$$
it is reasonable to generalize Conjecture \ref{Prime_var} as follows:

\begin{conj}
\label{APrime_covar}
For fixed $j,k \geq 1$ and $\epsilon \geq 0$,

$$
\frac{1}{X} \int_0^X \widetilde{\psi}_j(x;H)\, \widetilde{\psi}_k(x;H)\,dx \sim \frac{jk}{j+k-1} H \,(\log X - \log H)^{j+k-1},
$$

\vspace{2mm}
\noindent uniformly for $H\in[1,X^{1-\epsilon}]$.
\end{conj}

This is an estimate for the covariance of counts of almost-primes, and the simplicity of its form is closely related to the ratio conjecture of Farmer \cite{Fa}, a matter discussed at greater length in the appendix. It may be seen that it is possible to estimate the covariance of almost-primes counted by other weights, but these weights do not seem to produce so striking a pattern.

It is a combinatorial rearrangement of the prime number theorem that
$$
\frac{1}{X}\sum_{n\leq X} \Lambda_j(n)\Lambda_k(n) \sim \frac{jk}{j+k-1}\log^{j+k-1} X,
$$
in analogy of with \eqref{Covar_1_fin}.

In the author's thesis a variant of Conjecture \ref{APrime_covar}, with a weaker summability method, was shown to follow from the assumption that all correlations of the zeros of the zeta function locally follow the GUE (Gaussian Unitary Ensemble) random matrix prediction. In this connection, see also the earlier paper \cite{FaGoLeLe} of Farmer, Gonek, Lee, and Lester, who relate the covariance of almost-primes, weighted in a different manner than above, to the form factor of the zeta function's zeros.

It is reasonable to suppose that the statistics in Conjecture \ref{APrime_covar} are asymptotically normal as long as $H\rightarrow\infty$, and likewise in \eqref{Covar_3_fin}, but we do not pursue the matter here. (See \cite{MoSo} for a discussion of higher moments of $\widetilde{\psi}(x;\,H)$.)

\section{Preliminary material: Characters, explicit formulas, and equidistribution}

\subsection{}
In our proof of Theorem \ref{Covar_fin}, we will localize counts of almost-primes, following Keating and Rudnick, by applying harmonic analysis to a sum over Dirichlet characters. (In applying random matrix predictions to counts of almost-primes among the integers, one may make use instead of the characters $n^{-it}$.)

In the first half of this section we recall standard material explained succinctly in section 3 of \cite{KeRu}, and in greater length in, for instance, \cite{Ro}. In this section we adopt the notation of \cite{KeRu}. For $Q \in \bb F_q[T]$, $\Phi(Q)$ is defined as the order of the group $(\bb F_q[T] / Q)^\times$, that is the number of residues modulo $Q$ which are coprime to $Q$. A Dirichlet character modulo $Q$ is a function $\chi: \bb F_q[T] \rightarrow \bb C$ such that
\begin{enumerate}
\item $\chi(f) = \chi(g) \quad$ for all $f, g \in \bb F_q[T]$ such that $f \equiv g \MOD Q)$
\item $\chi(fg) = \chi(f)\chi(g) \quad$ for all $f,g \in \bb F_q[T]$.
\item $\chi(f) \neq 0\quad$ if and only if $(f,Q)\neq 1$.
\end{enumerate}
The number of Dirichlet characters modulo $Q$ is $\Phi(Q)$, and they satisfy the orthogonality relation
\begin{equation}
\label{Dirch_orth}
\frac{1}{\Phi(Q)} \sum_{\chi\MOD Q)} \overline{\chi(f)}\chi(g) = \begin{cases}
    1, & \text{if $f\equiv g \MOD Q)$ and $(fg, Q)=1$,}\\
    0, & \text{otherwise}.
  \end{cases}
\end{equation}
From this a standard argument shows that,
$$
\frac{1}{\Phi(Q)} \sum_{f \MOD Q)} \overline{\chi_1(f)} \chi_2(f) = \begin{cases}
    1, & \text{if $\chi_1 = \chi_2$,}\\
    0, & \text{otherwise,}
  \end{cases}
$$
though we will not need to reference this second fact in the argument that follows.

A Dirichlet character $\chi_2 \MOD Q_2)$ is said to be \emph{induced} by a Dirichlet character $\chi_1 \MOD Q_1)$, if $Q_1$ is a proper divisor of $Q_2$ and
$$
\chi_1(f) = [(f,Q_2)=1] \chi_2(f).
$$
Here, for a proposition $\mathfrak{A}$, we have used the notation,
$$
[\mathfrak{A}]:= 
\begin{cases}
    1, & \text{if $\mathfrak{A}$ is true,}\\
    0, & \text{if $\mathfrak{A}$ is false.}
  \end{cases}
$$

If $\chi_2$ is not induced by any characters, it is said to be a \emph{primitive} character. Note that if $Q_1$ is a proper divisor of $Q_2$ and $\chi_1$ is a character modulo $Q_1$, then $\chi_1$ induces exactly one character modulo $Q_2$, and so if $\Phi_{prim}(Q)$ is the number of primitive characters modulo a polynomial $Q$, then
$$
\Phi(Q) = \primesum_{D|Q} \Phi_{prim}(D),
$$
or by M\"{o}bius inversion,
$$
\Phi_{prim}(Q) = \primesum_{D|Q} \mu(D) \Phi(Q/D).
$$

A character $\chi$ is said to be \emph{even} if
$$
\chi(cf) = \chi(f)
$$
for all $c \in \bb F_q^\times.$ (This is in analogy with $\bb Z$, which has $\{-1,1\}$ as the group of units, and an even character $\chi$ is one that satisfies $\chi(-n) = \chi(n)$ for all $n$.)

The number $\Phi^{ev}(Q)$ of even Dirichlet characters modulo $Q$ is given by
$$
\Phi^{ev}(Q) = \frac{1}{q-1} \Phi(Q),
$$
which the reader may verify for himself or herself, using the fact that the abelian group $\bb F_q^\times$ has $q-1$ characters. As above, the number of primitive even characters modulo $Q$ is given by
$$
\Phi_{prim}^{ev}(Q) = \primesum_{D|Q} \mu(D) \Phi^{ev}(Q/D).
$$

The reader should have no trouble verifying that for $T^m \in \bb F_q[T]$,
\begin{equation}
\label{phi}
\Phi(T^m) = q^m(1-1/q)
\end{equation}
\begin{equation}
\label{phi_prim}
\Phi_{prim}(T^m) = q^m(1-1/q)^2
\end{equation}
\begin{equation}
\label{phi_ev}
\Phi^{ev}(T^m) = q^{m-1}
\end{equation}
\begin{equation}
\label{phi_evprim}
\Phi_{prim}^{ev}(T^m) = q^{m-1}(1-/q).
\end{equation}

\subsection{} 
The $L$-function associated with a Dirichlet character $\chi$ is defined for $|u| < 1/q$ by
\begin{align}
\label{euler}
\notag \mathcal{L}(u,\chi) &:= \primesum_f \chi(f)u^{\deg (f)} \\
& = \primeprod_P \frac{1}{1-\chi(P) u^{\deg(P)}},
\end{align}
where, again, the Euler product in the final line is over all monic irreducible polynomials. 

The trivial character $\chi_0 \MOD Q)$ is given by $\chi_0(f) = [(f,Q)=1]$. 

If $\chi$ is nontrivial, then it is standard (see \cite{Ro}) that $\mathcal{L}(u,\chi)$ is a polynomial, and may thus be analytically continued for $|u| \geq 1/q$. The content of the Riemann hypothesis for these $L$-functions, proved by Weil \cite{We}, is that all roots of $\mathcal{L}(u,\chi)$ lie on the circles $|u| = q^{-1/2}$ or $|u|=1$, the latter case being `trivial' zeros.\footnote{With language in closer correspondence to the classical Riemann hypothesis, $
L(s,\chi):=\mathcal{L}(q^{-s},\chi) = \primesum \chi(f)|f|^{-s}$ 
is zero only when $\Re s = 1/2$ or $0$.} As a well-known consequence of this, when $\chi$ is a non-trivial character modulo a polynomial $Q$ of degree $M$,
\begin{equation}
\label{RHbound}
\primesum_{\deg(f) = n} \Lambda(f)\chi(f) = O_M(q^{n/2}).
\end{equation}

\subsection{}
We can further describe the position of the roots of $\mathcal{L}(u,\chi)$ more conveniently if $\chi \MOD Q)$ is a primitive character. Specializing to this case, define
$$
\lambda_\chi := [\chi\text{ is even}].
$$
It is known that $\mathcal{L}(u,\chi)$ is a polynomial of degree $\deg Q -1$ in the variable $u$, and moreover that $\mathcal{L}(u,\chi)$ has a simple zero at $u=1$ if and only if $\chi$ is even. All of its zeros otherwise lie on the circle $|u| = q^{-1/2}$. Summarizing this information we have, for primitive characters $\chi$,
\begin{align}
\label{frob}
\notag \mathcal{L}(u,\chi) &= (1-\lambda_\chi u) \prod_{j=1}^N(1-q^{1/2} e^{i 2\pi \vartheta_j} u) \qquad \text{ for } N:= \deg Q - 1 - \lambda_\chi \\
&= (1 -\lambda_\chi u) \det(1 - q^{1/2} u \,\Theta_\chi), 
\end{align}
where $e^{i2\pi\vartheta_1}, ..., e^{i2\pi \vartheta_N}$ are constants on the unit circle which depend on the character $\chi$, and 
$$
\Theta_\chi := \textrm{diag}(e^{i2\pi\vartheta_1}, ..., e^{i2\pi \vartheta_N}).
$$
$\Theta_\chi$ is the `unitarized Frobenius matrix of $\chi$.' There is a more sophisticated realization of the operator $\Theta_\chi$ from which this name derives (see \cite{KaSa} for instance), but the above definition will be sufficient for our purposes.

By taking the logarithmic derivative of \eqref{euler} and \eqref{frob}, one obtains an explicit formula for primitive characters,
\begin{equation}
\label{explicit1}
\sum_{f\in\mathcal{M}_n} \Lambda(f)\chi(f) = -q^{n/2} \Tr (\Theta_\chi^n) - \lambda_\chi.
\end{equation}
The higher-order von Mangoldt functions are related to $\Theta_\chi$ in a similar manner, covered below.

A recent theorem of Katz (\cite{Ka}, Theorem 1.2) concerning the distribution of $\Theta_\chi$ will play a crucial role in our proof. As usual $PU(M-1)$ is the projective unitary group, the quotient of the unitary group $U(M-1)$ by unit modulus scalars, endowed with Haar measure, and we let $PU(M-1)^{\#}$ be the space of conjugacy classes of $PU(M-1)$, with inherited measure.

\begin{thm}[Katz]
\label{katz}
Fix $M\geq 3$. Over the family of even primitive characters $\chi \MOD T^{M+1})$, the conjugacy classes of the unitarized Frobenii $\Theta_\chi$ become equidistributed in $PU(M-1)^{\#}$ as $q\rightarrow\infty$.
\end{thm}

Said another way: for any continuous class function $\phi: U(M-1) \rightarrow \bb R$ such that $\phi(e^{i2\pi \theta} g) = \phi(g)$ for all unit scalars $e^{i2\pi \theta}$ and unitary matrices $g$, we have
$$
\lim_{q\rightarrow\infty} \frac{1}{\Phi_{prim}^{ev}(T^{M+1})} \sum_{\substack{\chi (T^{M+1}) \\ \text{prim., ev.}}} \phi(\Theta_\chi) = \int_{U(N-1)} \phi(g)\,dg.
$$

\subsection{} 
So far our presentation has followed that of \cite{KeRu} and the facts cited above would be enough to evaluate the variance of counts of primes. For almost-primes, we require an analogue of the explicit formula \eqref{explicit1}. To state the formula we first define, for a matrix $g$, the quantity $H_j^{(n)} = H_j^{(n)}(g)$ inductively as follows:
$$
H_1^{(n)}(g) = -\Tr(g^n),
$$
$$
H_j^{(n)} = \sum_{\substack{\ell+\lambda = n \\ \ell, \lambda \geq 1}} H_1^{(\lambda)} H_{j-1}^{(\ell)} + n H_{j-1}^{(n)}.
$$
The similarity to the inductive definition of $\Lambda_j$ should be clear. From \eqref{explicit1}, it is easy to see inductively that for primitive characters $\chi$ modulo a polynomial $Q$ of degree $M$,
\begin{equation}
\label{explicit2}
\sum_{f\in\mathcal{M}_n} \Lambda_j(f)\chi(f) = q^{n/2} H_j^{(n)}(\Theta_\chi) + O_{j,n, M}(q^{(n-1)/2}).
\end{equation}

If $\chi$ is not primitive, the above formula becomes a little more complicated to write, but at any rate inducting in the same way on \eqref{RHbound} we have for all $\chi \neq \chi_0$ modulo a polynomial $Q$ of degree $M$,
\begin{equation}
\label{RHbound2}
\sum_{f\in\mathcal{M}_n} \Lambda_j(f)\chi(f) = O_{j,n,M}(q^{n/2}).
\end{equation}

The quantities $H_j^{(n)}$ satisfy a somewhat surprising relation:
\begin{lem}
\label{H_Covar}
\begin{equation}
\label{H_covar}
\int_{U(N)} H_j^{(n)}(g) \overline{H_k^{(n)}(g)}\,dg = \sum_{d=1}^{n \wedge N} (d^j - (d-1)^j)(d^k - (d-1)^k).
\end{equation}
Additionally, if $n\neq m$,
$$
\int_{U(N)} H_j^{(n)} \overline{H_k^{(m)}}\, dg = 0.
$$
\end{lem}

We give a proof of Lemma \ref{H_Covar} in the appendix. Note that it generalizes a well known form-factor evaluation noted by Dyson \cite{Dy},
\begin{equation}
\label{Dyson}
\int_{U(N)} |\Tr(g^n)|^2\, dg = n \wedge N.
\end{equation}
Here $n \wedge N$ is the minimum of $n$ and $N$.

\section{A proof of Theorem \ref{Covar_fin}}

\subsection{}
We have already noted that, averaged over all monic polynomials $f$ of degree $n$, $\Lambda_j(f)$ has expected value $n^j - (n-1)^j$. This is a combinatorial consequence of the prime number theorem. We now turn to formula \eqref{Covar_1_fin}, and obtain as an easy corollary \eqref{Covar_2_fin}. These results too are combinatorial consequences of the prime number theorem, and in fact can be directly verified from \eqref{PNT_fin} without much effort for small $j$ and $k$. For $j$ and $k$ in general it is natural to make use of a generating series.

\begin{proof}[Proof of \eqref{Covar_1_fin} and \eqref{Covar_2_fin} in Theorem \ref{Covar_fin}]
For $|x| < 1/q$ we define
\begin{equation}
\label{Xi_1}
\Xi(s_1,s_2; x):= 1 + \sum_{j,k \geq 1} \frac{s_1^j}{j!} \frac{s_2^k}{k!} \primesum_f \Lambda_j(f) \Lambda_k(f) x^{\deg(f)} 
\end{equation}
Using \eqref{Lambda_j}, and letting $A = e^{s_1}, B = e^{s_2}$ with $s_1, s_2 \leq 0$, the right hand side may be simplified
$$
\primesum_f x^{\deg(f)} \primesum_{d|f} \mu(d) A^{\deg(f/d)} \primesum_{\delta|f} \mu(\delta) B^{\deg(f/\delta)}.
$$
As the summand is multiplicative in $f$, this is equal to
\begin{align*}
\primeprod_P \frac{1-x^{\deg(P)} A^{\deg(P)} - x^{\deg(P)} B^{\deg(P)} + x^{\deg(P)}}{1 - x^{\deg(P)}A^{\deg(P)}B^{\deg(P)}}.
\end{align*}
We note that,
$$
\frac{1-zu-zv+z}{1-zuv} = \frac{(1-zu)(1-zv)}{(1-zuv)(1-z)}\big(1+O(z^2)\big),
$$
for fixed $u$ and $v$. Using this for $z = x^{\deg(P)}, u = A^{\deg(P)}, v = B^{\deg(P)}$, we see that from the definition of $\mathcal{Z}(x)$ in \eqref{zeta_def} that
\begin{align*}
\Xi(s_1,s_2; x) &= \frac{\mathcal{Z}(x) \mathcal{Z}(xAB)}{\mathcal{Z}(xA)\mathcal{Z}(xB)} \primeprod_P\big(1+ O(x^{2 \deg(P)})\big) \\
& = \frac{(1-qxA)(1-qxB)}{(1-qx)(1-qxAB)}\big(1+O(q x^2)\big),
\end{align*}
so that
\begin{align}
\label{Xi_2}
\notag \lim_{q\rightarrow\infty} \Xi(s_1, s_2; x/q) &= \frac{(1-xA)(1-xB)}{(1-x)(1-xAB)} \\
&= 1 + \sum_{n=1}^\infty x^n \sum_{d\leq n}(A^d - A^{d-1})(B^d - B^{d-1}) \\
\notag &= 1+ \sum_{j,k\geq 1} \frac{s_1^j}{j!} \frac{s_2^k}{k!} \sum_{n=1}^\infty x^n \bigg(\sum_{d\leq n} (d^j - (d-1)^j)(d^k - (d-1)^k)\bigg).
\end{align}
From \eqref{vm_bound}, $\Lambda_j(f)\Lambda_k(f) \leq \deg^{j+k}(f),$ so that uniformly in $q$,
$$
\frac{1}{q^n} \sum_{f\in\mathcal{M}_n} \Lambda_j(f) \Lambda_k(f) \leq n^{j+k}.
$$
Returning to the definition \eqref{Xi_1}, a standard exercise in contour integration, making use of dominated convergence to pass to the limit \eqref{Xi_2}, implies that
$$
\lim_{q\rightarrow\infty}\frac{1}{q^n} \sum_{f\in\mathcal{M}_n} \Lambda_j(f) \Lambda_k(f) = \sum_{d\leq n} (d^j - (d-1)^j)(d^k - (d-1)^k),
$$
and this is \eqref{Covar_1_fin}.

From \eqref{vm_average} we derive
\begin{align*}
&\frac{1}{q^n} \sum_{f\in\mathcal{M}_n} \widetilde{\Lambda}_j(f) \widetilde{\Lambda}_k(f) \\
&= \frac{1}{q^n} \sum_{f\in\mathcal{M}_n} \Lambda_j(f) \Lambda_k(f) - (n^j-(n-1)^j)(n^k-(n-1)^k),
\end{align*}
and so \eqref{Covar_1_fin} immediately yields \eqref{Covar_2_fin}.
\end{proof}

\subsection{}
Any proof of \eqref{Covar_3_fin} must lie deeper. The main idea of the proof is to use Dirichlet characters modulo $T^{n-h}$ to localize sums over almost-primes. This entails a restriction to residues coprime to $T^{n-h}$, and it is for this reason that in the proof below we will frequently consider sums over polynomials $f$ such that $f(0)\neq 0$.

\begin{proof}[Proof of \eqref{Covar_3_fin} in Theorem \ref{Covar_fin}]
We outline the main steps of our proof, with details to follow.

\vspace{2.5mm}
\emph{Step 1:} We define
$$
E_j^\natural(n):= \frac{1}{|\mathcal{M}_n^\natural|} \primesum_{f \in \mathcal{M}_n^\natural} \Lambda_j(f) = \frac{1}{|\mathcal{P}_n^\natural|}\sum_{f\in\mathcal{P}_n^\natural} \Lambda_j(f),
$$
where for notational reasons we will write
$$
\mathcal{M}_n^\natural := \{f \in \mathcal{M}_n,\; f(0)\neq 0\}
$$
$$
\mathcal{P}_n^\natural := \{f:\; \deg(f) = n,\; f(0)\neq 0\},
$$
and
$$
\widetilde{\Lambda}_j^\natural(f):= \Lambda_j(f) - E_j^\natural(n),\quad \textrm{for $n := \deg(f)$},
$$
and
$$
\widetilde{\Psi}_j^\natural(f;\,h) := \sum_{\substack{g \in \mathcal{M}_n^\natural \\ g\in I(f;h)}} \widetilde{\Lambda}_j^\natural(f).
$$
(We will show later that $\widetilde{\Psi}_j \approx \widetilde{\Psi}_j^\natural$.)

Making use of an involution trick from \cite{KeRu}, we will see that
\begin{align}
\label{stepA}
&\sum_{f\in\mathcal{M}_n} \widetilde{\Psi}_j^\natural(f;\,h)\widetilde{\Psi}_j^\natural(f;\,h) \\
\notag &= \frac{q^{h+1}}{q-1} \sum_{g_1,g_2 \in \mathcal{P}_n^\natural} \big[ g_1 \equiv q_2\; (T^{n-h})\big]\, \widetilde{\Lambda}_j^\natural(g_1) \widetilde{\Lambda}_k^\natural (g_2).
\end{align}

\vspace{2.5mm}
\emph{Step 2:} On the other hand, using the orthogonality relation \eqref{Dirch_orth}, we will show that \eqref{stepA} is equal to 
\begin{equation}
\label{stepB}
q^{h+1}(q-1) \frac{1}{\Phi(T^{n-h})}\sum_{\substack{\chi \neq \chi_0\, (T^{n-h}) \\ \textrm{even}}} \bigg(\sum_{g_1\in\mathcal{M}_n} \Lambda_j(g_1) \chi(g_1)\bigg) \overline{\bigg(\sum_{g_2\in\mathcal{M}_n} \Lambda_k(g_2)\chi(g_2)\bigg)}.
\end{equation}

\vspace{2.5mm}
\emph{Step 3:} In turn, from \eqref{explicit2} and \eqref{RHbound2}, the expression \eqref{stepB} simplifies to
\begin{align}
\label{stepC}
&q^{h+1}(q-1)\frac{1}{q^{n-h}(1-q^{-1})}\sum_{\substack{\chi\,(T^{n-h}) \\ \textrm{ev., prim.}}} q^n H_j^{(n)}(\Theta_\chi) \overline{H_k^{(n)}(\Theta_\chi)} \\
&\notag + O_{n,h}\big(q^{n+h+1/2}\big).
\end{align}

\vspace{2.5mm}
\emph{Step 4:} From \eqref{phi_evprim} we see $(q-1)/q^{n-h}(1-q^{-1}) \sim \Phi_{prim}^{ev}(T^{n-h})$. The equidistribution Theorem \ref{katz} and the evaluation of Lemma \ref{H_Covar} thus show that \eqref{stepC} is asymptotic to
\begin{equation}
\label{stepD}
q^{h+1} q^n \sum_{d=1}^{n-h-2} (d^j-(d-1)^j)(d^k-(d-1)^k).
\end{equation}
so long as $n-h \geq 4$.

\vspace{2.5mm}
\emph{Step 5:} Finally, we show that $\widetilde{\Psi}^\natural$ is sufficiently close to $\widetilde{\Psi}$ that \eqref{stepD} implies \eqref{Covar_3_fin} in Theorem \ref{Covar_fin}.

\vspace{2.5mm}
We turn to the details.

\textbf{\textit{Step 1:}} The involution of Keating and Rudnick we will make use of is the mapping $f \mapsto f^\ast$ from the space $\mathcal{P}_n^\natural$ onto itself defined by
$$
(a_0 + a_1 T + \cdots a_n T^n)^\ast = a_n + a_{n-1} T + \cdots + a_0 T^n.
$$
Clearly if $f$ does not vanish at $0$,
$$
(f^\ast)^\ast = f,
$$ 
and for $f$ and $g$ non-vanishing at $0$,
$$
(fg)^{\ast} = f^\ast g^\ast.
$$
It follows that for $f\in \mathcal{P}_n^\natural$,
$$
\deg(f) = \deg(f^\ast),
$$
$$
\Lambda(f) = \Lambda(f^\ast),
$$
$$
\Lambda_j(f) = \Lambda_j(f^\ast),
$$
$$
\mu(f) = \mu(f^\ast),
$$
and so on.

This involution has the property that for $f_1, f_2 \in \mathcal{P}_n^\natural$, 
$$
\deg(f_1-f_2) \leq h
$$ 
if and only if 
$$
f_1^\ast - f_2^\ast \equiv 0 \MOD T^{n-h}),
$$ 
evident by comparing coefficients.

Hence
\begin{align*}
&\sum_{f\in \mathcal{M}_n} \widetilde{\Psi}_j^\natural(f;\,h)\widetilde{\Psi}_k^\natural(f;\,h) \\
&= \sum_{f\in\mathcal{M}_n} \sum_{f_1, f_2 \in \mathcal{M}_n^\natural} \big[\deg(f-f_1), \deg(f-f_2) \leq h\big] \,\widetilde{\Lambda}_j^\natural(f_1) \widetilde{\Lambda}_k^\natural(f_2) \\
&= q^{h+1} \sum_{f_1, f_2 \in \mathcal{M}_n^\natural} \big[\deg(f_1-f_2) \leq h\big] \widetilde{\Lambda}_j^\natural(f_1) \widetilde{\Lambda}_k^\natural(f_2) \\
&= \frac{q^{h+1}}{q-1} \sum_{f_1, f_2 \in \mathcal{P}_n^\natural}\big[\deg(f_1-f_2) \leq h\big] \widetilde{\Lambda}_j^\natural(f_1) \widetilde{\Lambda}_k^\natural(f_2) \\
&= \frac{q^{h+1}}{q-1} \sum_{g_1, g_2 \in \mathcal{P}_n^\natural} \big[ g_1 \equiv g_2\, (T^{n-h})\big]\, \widetilde{\Lambda}_j^\natural(g_1) \widetilde{\Lambda}_k^\natural(g_2),
\end{align*}
by reindexing with $g_1 = f_1^\ast$ and $g_2 = f_2^\ast$. This is \eqref{stepA}.

\textbf{\textit{Step 2:}} We note that for any $g \in \mathcal{P}_n^\natural$, the number of solutions $f \in \mathcal{P}_n^\natural$ to $f \equiv g\, (T^{n-h})$ is $q^{h+1}$. Writing $\widetilde{\Lambda}_j^\natural(f)$ as $\Lambda_j(f) - E_j^\natural(n)$ and expanding, we see that
\begin{align}
\label{covar_reduc}
\notag &\sum_{g_1, g_2 \in \mathcal{P}_n^\natural} \big[ g_1 \equiv g_2\, (T^{n-h})\big]\, \widetilde{\Lambda}_j^\natural(g_1) \widetilde{\Lambda}_k^\natural(g_2)\\
\notag &= \bigg(\sum_{g_1, g_2 \in \mathcal{P}_n^\natural} \big[ g_1 \equiv g_2\, (T^{n-h})\big]\, \Lambda_j(g_1) \Lambda_k(g_2)\bigg)
- q^{h+1}\, |\mathcal{P}_n^\natural| \,E_j(n) E_k(n) \\
&= I - II,
\end{align}
say.

Now,
$$
I = \frac{1}{\Phi(T^{n-h})}\sum_{\chi\, (T^{n-h})} \bigg(\sum_{\deg(g_1) = n} \Lambda_j(g_1) \chi(g_1)\bigg) \overline{\bigg(\sum_{\deg(g_2) = n} \Lambda_k(g_2) \chi(g_2)\bigg)}.
$$ 

This sum may be simplified with a result made use of in \cite{KeRu} as well, that
\begin{equation}
\label{charsum}
\sum_{c\in \bb F_q^\times} \chi(c) = (q-1)\,\big[\chi\textrm{ is even.}\big].
\end{equation}
\eqref{charsum} follows from the fact that
$$
\sum_{c\in \bb F_q^\times} \chi(c) = \sum_{c \in \bb F_q^\times} \chi(bc) = \chi(b)\sum_{c \in \bb F_q^\times} \chi(c),
$$
for any $b \in \bb F_q^\times$. Unless $\chi(b)=1$ for all $b\in \bb F_q^\times$, this implies the original sum is null.

Applying \eqref{charsum}, we see,
$$
I = \frac{(q-1)^2}{\Phi(T^{n-h})} \sum_{\substack{\chi\, (T^{n-h}) \\ \textrm{even}}} \bigg(\sum_{g_1\in\mathcal{M}_n} \Lambda_j(g_1) \chi(g_1)\bigg) \overline{\bigg(\sum_{g_2\in\mathcal{M}_n} \Lambda_k(g_2) \chi(g_2)\bigg)}.
$$
(Note that, having reduced the sum above over all characters to a sum of only even characters, we were able to further reduce the sums of all polynomials to sums of only monic polynomials.)

On the other hand, $|\mathcal{P}_n^\natural| = q^n(q-1)$ and $|\mathcal{M}_n^\natural| = q^{n-1}(q-1)$, so we have
$$
II = \frac{(q-1) q^{h+1}}{q^n} \sum_{g_1 \in \mathcal{M}_n} \Lambda_j(g_1) \chi_0(g_1) \overline{\primesum_{g_2\in \mathcal{M}_n} \Lambda_k(g_2) \chi_0(g_2)},
$$
where $\chi_0$ is the trivial character modulo $T^{n-h}$.

Because $\Phi(T^{n-h}) = q^{n-h}(1-q^{-1})$, this shows that \eqref{stepA} of Step 1 is equal to \eqref{stepB} of this Step 2.

It is worthwhile to reflect on this identity; its simplicity is obscured a little by these computations. The equality of \eqref{stepA} and \eqref{stepB} should not come as a complete surprise, since in both cases we are summing only the `oscillating part' of a series.

\textbf{\textit{Step 3:}} From equations \eqref{phi_ev} and \eqref{phi_evprim}, the number of non-primitive even characters modulo $T^{n-h}$ is $q^{n-h-1}$, and each inner sum of \eqref{stepB} over a non-primitive character we may bound with the Riemann hypothesis \eqref{RHbound2}. These terms collectively contribute an error term $O_{n,h}(q^{n+h})$ in \eqref{stepC}.

On the other hand, each inner sum in \eqref{stepB} over a primitive character may be written in terms of the Frobenii by \eqref{explicit2}, and this immediately yields \eqref{stepC}.

\textbf{\textit{Step 4:}} We need only check that for unitary matrices $g$, $\phi(g):=H_j^{(n)}(g)\overline{H_k^{(n)}(g)}$ is a class function satisfying $\phi(e^{i2\pi \theta} g) = \phi(g)$. This is straightforward.

\textbf{\textit{Step 5:}} We introduce the mapping of polynomials $f \mapsto f^{[i]}$ defined by
$$
(a_0 + a_1 T^1 + \cdots a_n T^n)^{[i]} := a_i + a_{i+1} T + \cdots a_n T^{n-i},
$$
so that if $T^i | f$,
$$
f = T^i f^{[i]}.
$$
We will make heavy use of the expansion \eqref{Lambda_fac} of $\Lambda_j(fg)$ to demonstrate
\begin{align}
\label{Psi_fac}
\widetilde{\Psi}_j(f;\,h) =& \;\widetilde{\Psi}_j^\natural(f;\,h) \\
\notag &+ \sum_{i=1}^h \sum_{\ell=0}^j {j \choose \ell} \,\Lambda_\ell(T^i)\, \widetilde{\Psi}_{j-\ell}^\natural(f^{[i]};\, h-i) + O_{j,n,h}(1).
\end{align}
All terms on the right hand side but the first will end up being negligible.

We begin by noting
$$
\widetilde{\Psi}_j(f;\,h) = \Psi_j(f;\,h) - \frac{q^{h+1}}{q^n} \sum_{f\in\mathcal{M}_n} \Lambda_j(g).
$$
Plainly,
\begin{align}
\label{red_res_A}
\notag \Psi_j(f;\,h) =& \sum_{\substack{g\in \mathcal{M}_n^\natural \\ \deg(g-f) \leq h}} \Lambda_j(g) + \sum_{\substack{g \in \mathcal{M}_{n-1}^\natural \\ \deg(Tg-f) \leq h}} \Lambda_j(Tg) + \cdots \\
&+ \sum_{\substack{g \in \mathcal{M}_{n-h}^\natural \\ \deg(T^h g -f) \leq h}} \Lambda_j(T^h g) + \underbrace{\Lambda_j(T^{h+1} f^{[h+1]})}_{= O_{j,n,h}(1)}.
\end{align}

On the other hand, it is easy to verify that for $0 \leq r \leq h \leq n = \deg(f)$,
\begin{equation*}
\frac{1}{|\mathcal{M}_{n-r}^\natural|} \sum_{\substack{g \in \mathcal{M}_{n-r}^\natural \\ \deg(T^r g -f) \leq h}} 1 = \frac{q^{h+1}}{q^n}.
\end{equation*}
Therefore
\begin{align}
\label{red_res_B}
\notag \frac{q^{h+1}}{q^n} \sum_{g\in\mathcal{M}_n} \Lambda_j(g) =& \frac{q^{h+1}}{q^n} \sum_{g \in \mathcal{M}_n^\natural} \Lambda_j(g) + \cdots \\
\notag &+ \frac{q^{h+1}}{q^n}\sum_{g \in \mathcal{M}_{n-h}^\natural} \Lambda_j(T^h g) + O_{j,n,h}(1) \\
=&  \sum_{\substack{g\in \mathcal{M}_n^\natural \\ \deg(g-f) \leq h}} \frac{1}{|\mathcal{M}_n^\natural|} \sum_{v \in \mathcal{M}_n^\natural} \Lambda_j(v) + \cdots \\
\notag &+ \sum_{\substack{g \in \mathcal{M}_{n-h}^\natural \\ \deg(T^h g -f) \leq h}}\frac{1}{|\mathcal{M}_{n-h}^\natural|} \sum_{v \in \mathcal{M}_{n-h}^\natural} \Lambda_j(T^h v) + O_{j,n,h}(1).
\end{align}
By definition,
$$
\frac{1}{|\mathcal{M}_n^\natural|} \sum_{v \in \mathcal{M}_n^\natural} \Lambda_j(v) = E_n^\natural,
$$
and by using \eqref{Lambda_fac} to expand $\Lambda_j(T^i v)$, we have for $v(0) \neq 0$
$$
\frac{1}{|\mathcal{M}_{n-i}^\natural|} \sum_{v \in \mathcal{M}_{n-i}^\natural} \Lambda_j(T^i v) = \sum_{\ell=0}^j {j \choose \ell} \Lambda_\ell(T^i) E_{j-\ell}^\natural(n-i).
$$

A similar expansion can be performed in \eqref{red_res_A}. We subtract \eqref{red_res_B} from \eqref{red_res_A} and note that $\deg(T^i g-f) \leq h$ if and only if $\deg(g - f^{[i]}) \leq h-i$. This gives us the identity \eqref{Psi_fac} that we were after.

We have shown in steps 1 through 4 that for fixed $n,h$ and $j,k$
$$
\sum_{f\in\mathcal{M}_n} \widetilde{\Psi}_j^\natural(f;\,h) \widetilde{\Psi}_k^\natural(f;\,h) \sim q^{h+1} q^n \sum_{d=1}^{n-h-2} (d^j-(d-1)^j)(d^k-(d-1)^k).
$$
This implies in particular that for $0 \leq i \leq h$,
\begin{align*}
\sum_{f\in\mathcal{M}_n} \widetilde{\Psi}_j^\natural(f^{[i]};\, h-i)^2 &= q^i \sum_{f\in\mathcal{M}_{n-i}} \widetilde{\Psi}_j^\natural(f;\,h-i)^2 \\
&= O_{n,h,j}(q^{h+1} q^n / q^i),
\end{align*}
and so by Cauchy-Schwarz for $0 \leq i_1, i_2 \leq h$,
$$
\sum_{f\in\mathcal{M}_n} \widetilde{\Psi}_j^\natural(f^{[i_1]};\,h-i_1) \widetilde{\Psi}_k^\natural(f^{[i_2]};\,h-i_2) = O_{n,h,j,k}(q^{h+1} q^n / q^{(i_1+i_2)/2}).
$$
This is $O(q^{h+1} q^n / q^{1/2})$ as long as one of $i_1$ or $i_2$ are non-zero. Hence, using the expansion \eqref{Psi_fac}, for fixed $n,h$ and $j,k$ we have
\begin{align*}
&\frac{1}{q^{h+1}}\frac{1}{q^n} \sum_{f\in\mathcal{M}_n} \widetilde{\Psi}_j(f;\,h) \widetilde{\Psi}_k(f;\,h) \\
&= \frac{1}{q^{h+1}}\frac{1}{q^n} \sum_{f\in\mathcal{M}_n} \widetilde{\Psi}_j^\natural(f;\,h) \widetilde{\Psi}_k^\natural(f;\,h) + O(1/\sqrt{q}) \\
& \sim \sum_{d=1}^{n-h-2} (d^j-(d-1)^j)(d^k-(d-1)^k),
\end{align*}
as claimed.
\end{proof}

\textbf{Acknowledgements:}
I thank Paul-Olivier Dehaye, Steve Gonek, Emmanuel Kowalski, Hugh Montgomery, Igor Pak, Terence Tao, and especially Jon Keating and Zeev Rudnick, for helpful discussions and comments. It was a suggestion of Rudnick's to apply the results of my thesis to a function field setting, where they might be proved unconditionally. I also thank the anonymous referee for a careful reading of the manuscript, with several helpful suggestions and corrections. The author was partially supported by an AMS-Simons Travel Grant

\begin{appendix}
\section{Some random matrix statistics}

\subsection{} 
In this appendix we give a proof of Lemma \ref{H_Covar}, evaluating the averages of the statistics $H_j^{(n)} \overline{H_k^{(m)}}$. We do this by decomposing $H_j^{(n)}$ into a linear combination of Schur functions, which play a fundamental role in the representation theory of the unitary group. This proof is adapted from material that first appeared in the author's thesis.

We recall without proof some essential facts from symmetric function theory. A more complete introduction with proofs of the facts cited below is found in \cite{Bu}. The references \cite{St} and \cite{Ga} are also useful, the latter being a streamlined introduction from the perspective of random matrices and analytic number theory.

In the variables $\omega_1,...,\omega_N$, recall the definitions that for $k=0,1...,N$
$$
e_k = e_k(\omega_1,...,\omega_N):= \sum_{j_1 < \cdots < j_k} \omega_{j_1}\cdots\omega_{j_k}
$$
with $e_o:=1$, while for $k=0,1,...$,
$$
h_k = h_k(\omega_1,...,\omega_N):= \sum_{j_1\leq \cdots \leq h_k} \omega_{j_1}\cdots\omega_{j_k}.
$$
with $h_0:= 1$.

A partition $\lambda = (\lambda_1,\lambda_2,...)$ is a sequence of non-negative integers $\lambda_1\geq \lambda_2 \geq ...$ such that for large enough $n$, $\lambda_{n+1}=0$. $\lambda$ may then be thought of as just $(\lambda_1,...,\lambda_n)$, and the largest $n$ such that $\lambda_n \neq 0$ is called the \textit{length} of $\lambda$.

If the length of $\lambda$ is no more than $N$, we define the Schur function $s_\lambda$ by
$$
s_\lambda= s_\lambda(\omega_1,...,\omega_N):= \frac{\det_{N\times N}\big( \omega_i^{\lambda_j+N-j}\big)}{\det_{N\times N}\big(\omega_i^{N-j}\big)}.
$$
The functions also satisfy
$$
s_\lambda = \sum_T  (\omega_{t[1,1]}\omega_{t[1,2]}\cdots\omega_{t[1,\lambda_1]}) \cdots (\omega_{t[2,1]}\cdots\omega_{t[2,\lambda_2]})(\omega_{t[n,1]}\cdots \omega_{t[n,\lambda_n]}),
$$
where $n$ is the length of $\lambda$ and the sum is over all so-called \textit{semi-standard Young tableau} of shape $\lambda$, numbers
\[
\begin{array}{lllll}
t[1,1] & t[1,2] & \dots & \dots & t[1,\lambda_1] \\
t[2,1] & t[2,2] & \dots & t[2,\lambda_2] & \\
\vdots & \vdots & \ddots & &  \\
t[n,\lambda_n] & \dots & t[n,\lambda_n] & & \\
\end{array}
\]
with $t[i,j] \in \{1,2,...,N\}$ for all $i,j$, so that in rows numbers from left-to-right are non-decreasing:
$$
t[i,1] \leq t[i,2] \leq \cdots \leq t[i,\lambda_i],
$$
while in columns
\begin{align*}
&\hspace{2mm} t[1,j] \\
&< t[2,j] \\
&\hspace{4mm}\vdots \\
&< t[\cdot,j]
\end{align*}
numbers are strictly increasing. For instance, when $N=3$, the semi-standard Young tableaux of the partition $(2,1)$ are

\vspace{4mm}

\ytableausetup{mathmode, boxsize=1.1em}
\ytableaushort{11,2}
\hspace{3mm}
\ytableausetup{nobaseline}
\ytableaushort{12,2}
\hspace{3mm}
\ytableausetup{nobaseline}
\ytableaushort{13,2}
\hspace{3mm}
\ytableausetup{nobaseline}
\ytableaushort{11,3}
\hspace{3mm}
\ytableausetup{nobaseline}
\ytableaushort{12,3}
\hspace{3mm}
\ytableausetup{nobaseline}
\ytableaushort{13,3}
\hspace{3mm}
\ytableausetup{nobaseline}
\ytableaushort{22,3}
\hspace{3mm}
\ytableausetup{nobaseline}
\ytableaushort{23,3}

\vspace{4mm}

\noindent and
$$
s_{(2,1)}(\omega_1,\omega_2, \omega_3) = \omega_1^2 \omega_2 + \omega_1 \omega_2^2 + \omega_1^2 \omega_3 + \omega_1 \omega_3^2 + \omega_2^2 \omega_3 + \omega_2 \omega_3^2 + 2\omega_1 \omega_2 \omega_3.
$$

For us the importance of Schur functions is that
\begin{equation}
\label{orthoschur}
\int_{U(N)}s_{\lambda_1}(\omega_1,...,\omega_N)\overline{s_{\lambda_2}(\omega_1,...,\omega_N)}\,dg = \delta_{\lambda_1 = \lambda_2},
\end{equation}
for all partitions $\lambda_1,\lambda_2$ of length no more than $N$, where $\omega_1,...,\omega_N$ are the eigenvalues of $g\in U(N)$; a proof of this fact can be found in \cite{Bu} or \cite{Ga}.

Finally, let us introduce the abbreviation
$$
\lambda = (\lambda_1,...,\lambda_j,\underbrace{1,...,1}_{k}) = (\lambda_1,...,\lambda_j, 1^k).
$$
This generalizes in the obvious way, but the above usage is all that we will make use of.

It will be convenient to have defined the characteristic polynomial of $g\in U(N)$, with eigenvalues $\omega_1,...,\omega_N$,
$$
Z(\beta):= \det(1-e^{-\beta}g) = \prod_{j=1}^N (1-e^{-\beta}\omega_j).
$$
By logarithmic differentiation, we have
$$
\frac{Z'}{Z}(\beta) = \sum_{n=1}^\infty e^{-\beta n} \Tr(g^n),
$$
and because
$$
(-1)^j \frac{Z^{(j)}}{Z} = \Big(-\frac{Z'}{Z} - \frac{d}{d\beta}\Big) \Big( (-1)^{j-1} \frac{Z^{(j-1)}}{Z}\Big),
$$
it follows inductively that
\begin{equation}
\label{highlogderivZ}
(-1)^j\frac{Z^{(j)}}{Z}(\beta) = \sum_{r=1}^\infty e^{-\beta r} H_j^{(r)}.
\end{equation}
(This should not be surprising, in light of the well-known identity
$$
(-1)^j \frac{\zeta^{(j)}}{\zeta}(s) = \sum_{n=1}^\infty \frac{\Lambda_j(n)}{n^s},
$$
and its function field analogues.)

It is easy to check that,
\begin{equation}
\label{Zto_e}
Z(\beta) = \sum_{n=0}^N (-1)^n e_n e^{-\beta n}
\end{equation}
\begin{equation}
\label{Zjto_e}
Z^{(j)}(\beta) = (-1)^j \sum_{n=0}^N (-1)^n n^j e_n e^{-\beta n}
\end{equation}
\begin{equation}
\label{Zto_h}
\frac{1}{Z(\beta)} = \sum_{m=0}^\infty h_m e^{-\beta m}.
\end{equation}

\begin{proof}[Proof of Lemma \ref{H_Covar}]
All symmetric functions in this section are in the variables $\omega_1,...,\omega_N$, eigenvalues of a unitary matrix. We show that
\begin{equation}
\label{HjtoSchur}
H_j^{(r)} = \sum_{\nu=1}^{r \wedge N} (-1)^\nu (\nu^j - (\nu-1)^j) s_{(r-\nu+1,1^{\nu-1})}.
\end{equation}
Lemma \ref{H_Covar} then follows from the Schur orthogonality relation \eqref{orthoschur}.

In the first place, from \eqref{highlogderivZ} and \eqref{Zjto_e} and \eqref{Zto_h}, we have by pairing coefficients,
\begin{equation}
\label{Hto_eh}
H_j^{(r)} = \sum_{\nu=1}^{r\wedge N} (-1)^\nu \nu^j e_\nu h_{r-\nu}.
\end{equation}
But note that for $n\geq 1$, $m \geq 0$,
$$
e_n h_m = \sum_{\alpha_1 < \cdots < \alpha_n} \sum_{\beta_1 \leq \cdots \leq \beta_m} \omega_{\alpha_1}\cdots\omega_{\alpha_n}\omega_{\beta_1}\cdots \omega_{\beta_m},
$$
where if $m=0$ the sum over $\beta$ is understood to be empty.

In the case that $m\neq 0$, breaking the sum into two parts depending on whether $\alpha_1 \leq \beta_1$ or $\beta_1 < \alpha_1$, this sum is
\begin{align*}
e_n h_m =& \sum_{\alpha_1\leq \beta_1 \cdots \leq \beta_m} \omega_{\alpha_1}\omega_{\beta_1}\cdots \omega_{\beta_n} \sum_{\substack{\alpha_2, ..., \alpha_n \\ \textrm{ such that } \\ \alpha_1 < \alpha_2 < \cdots < \alpha_n}} \omega_{\alpha_2}\cdots\omega_{\alpha_n} \\
&+ \sum_{\beta_1 \leq \cdots \leq \beta_m} \omega_{\beta_1}\cdots\omega_{\beta_m} \sum_{\substack{\alpha_1,...,\alpha_n \\ \textrm{ such that } \\ \beta_1 < \alpha_1 < \cdots < \alpha_n}} \omega_{\alpha_1}\cdots\omega_{\alpha_n} \\
=& s_{(m+1,1^{n-1})} + s_{(m,1^n)}
\end{align*}
Provided we adopt the convention that $s_{(0,1^n)} = 0$, this remains true when $m=0$.

On the other hand, if $n=N$,
\begin{align*}
e_n h_m &= \omega_1\cdots\omega_N \sum_{\beta_1 \leq \cdots \leq \beta_m} \omega_{\beta_1}\cdots\omega_{\beta_m} \\
&= s_{(m+1,1^{N-1})}
\end{align*}
since in this case, for any indices $\beta_1,...,\beta_m$ of the sum, $1\leq \beta_1$.

\vspace{3mm}
\noindent\textit{Remark:} The above identities are a special case of the well known Pieri rule.
\vspace{3mm}

Therefore, for all $j \geq 1$ and $r \geq 1$,
\begin{align*}
H_j^{(r)} & = \sum_{\nu=1}^{r \wedge N} (-1)^\nu \nu^j \big(s_{(r-\nu+1,1^{\nu-1})} + \delta_{\nu \neq r,N} s_{(r-\nu,1^\nu)}\big) \\
&= \sum_{\nu=1}^{r \wedge N} (-1)^\nu (\nu^j - (\nu-1)^j) s_{(r-\nu+1,1^{\nu-1})},
\end{align*}
as claimed. Applying \eqref{orthoschur} to this proves the lemma.
\end{proof}

\subsection{} 
We have mentioned that the algebraic simplicity of Theorem \ref{Covar_fin} and Conjecture \ref{APrime_covar} is closely connected to a ratio conjecture of Farmer for autocorrelations of the zeta function. We elaborate on this point by showing that Lemma \ref{H_Covar} is equivalent to the $2\times 2$ ratio theorem for the unitary group. Other earlier proofs of this result, and its generalization to $n\times m$ ratios, may be found in \cite{BuGa}, \cite{CoFaZi}, and \cite{CoFoSn}. The first of these makes use of a Schur function decomposition also, though the authors arrange their combinatorics differently.

\begin{thm}
\label{ratiotheorem}
For $A, B, C, D$ complex numbers with $|C|, |D| < 1$, for $N \geq 1$,
\begin{align}
\label{ratiothm}
&\int_{U(N)}\frac{\det(1-Au)\det(1-Bu^{-1})}{\det(1-Cu)\det(1-Du^{-1})}\,du \\
\notag &= \frac{(1-BC)(1-AD)}{(1-AB)(1-CD)} + (AB)^N \frac{(1-CA^{-1})(1-DB^{-1})}{(1-(AB)^{-1})(1-CD)}.
\end{align}
\end{thm}

\begin{proof}
We let
$$
A=e^{-\beta_1+s_1}
$$
$$
B=e^{-\beta_2+s_2}
$$
$$
C = e^{-\beta_1}
$$
$$
D = e^{-\beta_2}
$$
with $\Re \beta_1, \Re \beta_2 > 0$. (There is no real loss of generality to assume $A,B,C$ and $D$ are real, and this would make less to keep track of in the argument that follows if the reader desires.)

Then the left hand side of \eqref{ratiothm} is
\begin{align*}
&\int_{U(N)} \frac{Z(\beta_1-s_1)}{Z(\beta_1)} \overline{\bigg(\frac{Z(\overline{\beta_2}-\overline{s_2})}{Z(\overline{\beta_2})}\bigg)}\,du \\
&= \sum_{j,k=0}^\infty \frac{s_1^j}{j!}\frac{s_2^k}{k!} \int_{U(N)}\bigg((-1)^j\frac{Z^{(j)}}{Z}(\beta_1)\bigg)\overline{\bigg((-1)^k \frac{Z^{(k)}}{Z}(\overline{\beta_2})\bigg)}\,du \\
&= 1 + \sum_{j,k=1}^\infty \sum_{r,s=1}^\infty \frac{s_1^j}{j!}\frac{s_2^k}{k!} e^{-r\beta_1-s\beta_2} \int_{U(N)}H_j^{(r)}\overline{H_k^{(s)}}\,dg.
\end{align*}
Here we have used that,
$$
\int_{U(N)} \frac{Z^{(0)}}{Z}(\beta_1)\overline{\frac{Z^{(0)}}{Z}(\overline{\beta_2})}\,du = 1
$$
and slightly less trivially that
$$
\int_{U(N)}\frac{Z^{(j)}}{Z}(\beta)\,du = 0
$$
for $j \geq 1$, which follows from, for instance, equations \eqref{Zjto_e}, \eqref{Zto_h} (with all exponents of $\omega$ being positive in both identities) and that for any $\theta$, $u\mapsto e^{i2\pi \theta} u$ preserves the Haar measure of the unitary group.

On the other hand, after rearranging the right hand side of \eqref{ratiothm}, it is just
\begin{align*}
&1 + \frac{(1-e^{-s_1})(1-e^{-s_2})}{1-e^{-\beta_1-\beta_2}}e^{s_1+s_2-\beta_1-\beta_2}\bigg(\frac{1-(e^{s_1+s_2-\beta_1-\beta_2})^N}{1-e^{s_1+s_2-\beta_1-\beta_2}}\bigg) \\
&= 1 + \frac{(1-e^{-s_1})(1-e^{-s_2})}{1-e^{-\beta_1-\beta_2}}\sum_{\nu=1}^N e^{\nu(s_1+s_2)} e^{-\nu(\beta_1+\beta_2)} \\
&= 1 + (1-e^{-s_1})(1-e^{-s_2}) \sum_{\nu=1}^N e^{\nu(s_1+s_2)}\sum_{r=\nu}^\infty e^{-(\beta_1+\beta_2)r} \\
& = 1 + \sum_{r=1}^\infty e^{-(\beta_1+\beta_2)r} \sum_{\substack{\nu \leq r \\ \nu \leq N}} \big[e^{\nu s_1}- e^{(\nu-1) s_1}\big] \big[e^{\nu s_2} - e^{(\nu-1)s_2}\big] \\
&= 1 + \sum_{j,k=1}^\infty \frac{s_1^j}{j!} \frac{s_2^k}{k!} \sum_{r=1}^\infty e^{-(\beta_1+\beta_2)r} \sum_{\nu=1}^{r \wedge N} \big(\nu^j - (\nu-1)^j\big)\big(\nu^k - (\nu-1)^k\big).
\end{align*}
By pairing coefficients, Lemma \ref{H_Covar} is therefore equivalent to the right and left hand sides of \eqref{ratiothm} being equal.
\end{proof}

\subsection{} 
Working over $\bb F_q[T]$, we have seen in the explicit formula \eqref{explicit1} that counts of primes by the von Mangoldt function, twisted by a random character, correspond to the trace of a random matrix. Likewise, counts of almost-primes correspond to the quantities $H_j^{(n)}$, in the sense of \eqref{explicit2}. (A correspondence of this sort remains true in the number field setting as well, at least under random matrix predictions for the zeros of the zeta function and other L-functions.) 

It is therefore natural to ask what arithmetic function corresponds to the Schur functions, and at least for hook Schur functions the answer is not complicated: they correspond to a truncated sum of the M\"{o}bius function over divisors. 

More precisely, define
$$
\delta_m(f):= \primesum_{\substack{d|f \\ \deg(d) < m}} \mu(d).
$$
Then for a primitive character $\chi \MOD T^{M+1})$ with Frobenius matrix $\Theta_\chi$,
\begin{equation}
\label{schur}
\frac{(-1)^k}{q^{(m+k)/2}}\sum_{f \in \mathcal{M}_{m+k}}\delta_m(f)\chi(f) = s_{(m,1^k)}(\Theta_\chi) + O_{m,k,M}(q^{-1/2}),
\end{equation}
with the convention that 
$$
s_{(m,1^k)}=0, \quad \textrm{ for } k \geq M -\lambda_\chi.
$$
In fact, for $\Theta_\chi$ an odd character, \eqref{schur} holds with equality.

Schur functions that do not correspond to a hook partition can be generated from those that do by means of, for instance, the Giambelli identity (Exercise 7.39 in \cite{St}).

There are many way to verify \eqref{schur}, none very arduous. We take a route whose starting point is the explicit formula \eqref{explicit1} and its generalization \eqref{explicit2}. This method is slightly more roundabout than necessary, but we take it because the explicit formula has an elegant counterpart in the number field setting.

Note that
\begin{align}
\label{lin_comb}
\sum_{f\in \mathcal{M}_n} \chi(f)\Lambda_j(f) &= \sum_{f\in \mathcal{M}_n} \primesum_{d | f} \mu(d) \deg^j(f/d) \\
\notag &= \sum_{\nu=1}^n (-1)^\nu\big(\nu^j - (\nu-1)^j\big)\bigg((-1)^{\nu-1} \sum_{f\in\mathcal{M}_n}\chi(f)\delta_{n-\nu+1}(f)\bigg).
\end{align}
The left hand side of \eqref{lin_comb} is equal to $q^{n/2} H_j^{(n)}(\Theta_\chi)$ plus a small error term, while the right hand side corresponds exactly to the formula \eqref{HjtoSchur}. Because $j$ is free while $\nu$ ranges over only a finite number of places, we may compare coefficients to obtain \eqref{schur}.

This correspondence suggests that the simple algebraic form of Theorem \eqref{Covar_fin} and Conjecture \eqref{APrime_covar} is due more to the presence of the M\"{o}bius function in the formulas for the higher von Mangoldt weights \eqref{Lambda_j} than anything else. 

Indeed, if we allow ourselves to speculate, it seems likely that the arithmetic counterpart to the Schur orthogonality relation \eqref{orthoschur} is in some way a consequence of the conjectured randomness of the M\"{o}bius function,\footnote{See \cite{Sa} for an introduction to some aspects of this, along with \cite{CaRu} for recent progress in a function field setting.} but we cannot offer here any precise statement of this sort. 

\end{appendix}

\bibliographystyle{plain}

\end{document}